\documentclass[12pt, psamsfonts]{amsart}
\usepackage{graphicx}
\usepackage{float}
\usepackage{scrextend}
\usepackage{breqn}
\usepackage{amsmath,amsthm,amsfonts,amssymb}
\usepackage{eucal}
\usepackage[all,knot]{xy}
\usepackage{thmtools}
\usepackage{thm-restate}
\xyoption{arc}
\usepackage{tikz}
\usetikzlibrary{matrix}
\usepackage{paralist}
\usepackage{subcaption}
\usepackage[section]{placeins}
\newcommand{\D}{\mathcal{D}}

\newcommand{\mZ}{\mathcal{Z}}
\newcommand{\Z}{\mathbb{Z}}
\newcommand{\ot}{\otimes}

\newcommand{\CC}{\mathcal{C}}
\renewcommand{\D}{\mathcal{D}}
\renewcommand{\Z}{\mathbb{Z}}
\renewcommand{\1}{\textbf{1}}
\newcommand{\LL}{\mathcal{L}}
\DeclareMathOperator{\Rep}{Rep}

\newtheorem{theorem}{Theorem}[section]
\newtheorem{thm}{Theorem}
\newtheorem{proposition}[theorem]{Proposition}

\newtheorem{lemma}[theorem]{Lemma}
\newtheorem{definition}[theorem]{Definition}

\date{\today}

\begin{document}
\title[Integral Metaplectic Modular Categories]{Integral Metaplectic Modular Categories}

\author{Adam Deaton}
\email{deaton@math.tamu.edu}
\address{Department of Mathematics\\
    Texas A\&M University \\
    College Station, TX 77843\\
    U.S.A.}

\author{Paul Gustafson}
\email{pgustafs@math.tamu.edu}
\address{Department of Mathematics\\
    Texas A\&M University \\
    College Station, TX 77843\\
    U.S.A.}

\author{Leslie Mavrakis}
\email{mavrakisl@spu.edu}
\address{Department of Mathematics\\
    Seattle Pacific University\\
    Seattle, WA 98119-1997\\
    U.S.A.}

\author{Eric C. Rowell}
\email{rowell@math.tamu.edu}
\address{Department of Mathematics\\
    Texas A\&M University \\
    College Station, TX 77843\\
    U.S.A.}

\author{Sasha Poltoratski}
\email{sashapolt@gmail.com}
\address{Department of Mathematics\\
    University of California Berkeley\\
    Berkeley, CA 94720
    U.S.A.}

\author{Sydney Timmerman}
\email{stimmer2@jhu.edu}
\address{Department of Physics and Astronomy\\
    Johns Hopkins University\\
    Baltimore, MD 21218-2686\\
    U.S.A.}
    
\author{Benjamin Warren}
\email{bwarren1@swarthmore.edu}
\address{Department of Mathematics and Statistics\\
    Swarthmore College\\
    Swarthmore, PA 19081\\
    U.S.A.}
    
\author{Qing Zhang}
\email{zhangqing2513@tamu.edu}
\address{Department of Mathematics\\
    Texas A\&M University \\
    College Station, TX 77843\\
    U.S.A.}
\thanks{This research was carried out as part of an REU at Texas A\&M University in which  L.M., S.P., S.T. and B.W. participated.  The research of E.C.R., A.D., P.G. and Q.Z. was partially supported by NSF grant DMS-1664359.}
\renewcommand{\shortauthors}{Deaton, Gustafson, Mavrakis, Rowell, Poltoratski, 
 Timmerman, Warren, Zhang}

\begin{abstract} A braided fusion category is said to have Property \textbf{F} if the associated braid group representations factor over a finite group.
We verify integral metaplectic modular categories have property \textbf{F} by showing these categories are group theoretical.    For the special case of integral categories $\CC$ with the fusion rules of $SO(8)_2$ we determine the finite group $G$ for which $Rep(D^{\omega}G)$ is braided equivalent to $\mZ(\CC)$.  In addition, we determine the associated classical link invariant, an evaluation of the 2-variable Kauffman polynomial at a point. 
\end{abstract}
\maketitle

\section{Introduction}

Bosonic topological phases of matter can be modeled by unitary modular categories \cite{RWBULL}, which may be regarded as the programming language for topological quantum computation.  Physically speaking, every simple object corresponds to an anyon type, and the braid group representations provide quantum gates.  When the braid group representation associated with each anyon type has finite image, a universal topological quantum computer cannot be constructed from any anyon type in $\CC$ using braiding alone.  It is an important problem to characterize which categorical models admit such universal braiding gates, and which do not.  A category is said to have Property \textbf{F} if the braid group representations associated with every simple object factor over a finite group \cite{NR}.  
The Property \textbf{F} conjecture states that a category has Property \textbf{F} if and only if it is weakly integral \cite{NR}, i.e. if and only if every simple object $X$ has $\dim(X)^2\in\Z$.

Unitary modular categories with the same fusion rules as $SO(N)_2$ for integer $N > 1$, or metaplectic modular categories, form an important class of (weakly integral) modular categories. As noted in \cite{ACRW}, the study of metaplectic modular categories first arose from a generalization of Majorana zero modes.  We focus on the case of integral metaplectic modular categories\textemdash those for which the dimension of every simple object is an integer.

A category $\CC$ that is group theoretical, in some sense, \textquotedblleft comes from" a finite group; the Drinfeld center $\mZ (\CC)$ is braided equivalent to the representation category of the twisted Drinfeld double of a finite group $G$, $\Rep(D^w G)$ \cite{ENO}.  Most importantly, group theoretical categories have Property \textbf{F} \cite{ERW}.     
\begin{restatable}{thm}{GT}
All integral metaplectic modular categories $\CC$ are group theoretical.  
\end{restatable}
In particular, our work verifies the Property \textbf{F} conjecture for integral metaplectic modular categories.  

Although this may seem to be a negative result from the perspective of topological quantum computation, there are two very compelling reasons to study (weakly) integral modular categories: 1) the non-Abelian anyons that are on everyone's minds these days as being physically realizable are the Majorana anyons, which are weakly integral and 2) universality may often be recovered by including non-topological operations \cite{C}.

The group theoreticity of integral metaplectic modular categories suggests a couple further directions of investigation.  One direction is determining the finite group $G$ for which $\mZ (\CC) = \Rep(D^w G)$.  Another direction of particular interest is determining the link invariant associated with each of these categories.  The extended Property \textbf{F} conjecture states that a category has Property \textbf{F} if and only if the associated link invariants are classical \cite{R}. Any topological quantum computation using anyons from these categories will approximate their associated link invariants evaluated on a braid closure.  Thus, understanding the link invariant is equivalent to understanding how a topological quantum computer created using these anyon systems will behave. 

We examine the specific case of modular categories with the same fusion rules as $SO(8)_2$. This is the smallest case of an integral metaplectic modular category for which the Drinfeld center cannot be equivalent to the representation category of the untwisted double of any finite group; i.e., there must be a nontrivial cocycle. Additionally, all simple objects are of dimension either one or two, creating extra symmetry and removing ambiguity regarding the associated link invariant. 

\begin{thm}
The link invariant associated with a category with fusion rules of $SO(8)_2$ is:
$$\frac{1}{2}(-1)^{w(D)}\sum_{X\subset L}(-i)^{\text{linking number}(X,L-X)}$$ where $L$ is a link, $D$ is its diagram, $w(D)$ is the writhe of $D$, and the sum is over components $X\subset L$.
\end{thm}

In section two, we prove the group theoreticity of integral metaplectic modular categories. Section three will contain our examination of categories with fusion rules of $SO(8)_2$.

\section{Group Theoreticity}
Throughout we will assume that the reader is familiar with basic categorical notions, and refer to \cite{BK} and \cite{ENO} for the standard definitions and details. 
Before we proceed to our results we need some standard notation.  Let $S$ denote the $S$ matrix of a braided fusion category $\CC$.  
We say an objects $X,Y$ in a braided fusion category $\CC$ \emph{centralize each other} if $S_{XY} = \dim(X)\dim(Y)$, in other words the double braiding between $X$ and $Y$ is trivial: $c_{Y,X}c_{X,Y}=Id_{X\otimes Y}$.  Strictly speaking this is a theorem in \cite{M} characterizing the latter condition in terms of the $S$-matrix.
For any subcategory $\LL$ of a braided fusion category $\CC$ the \textit{centralizer} of $\LL$ denoted by $\mZ_{\CC}(\LL)$ is the subcategory consisting of objects $Y \in \CC$ that centralize all objects $X \in \LL$.  Let $\mathcal{L} \subset \mathcal{C}$. Then, the \emph{adjoint subcategory} $\mathcal{L}_{ad}$ is the smallest fusion subcategory of $\mathcal{C}$ that contains $X \ot X^{*}$ for each simple object X $\in \mathcal{L}$, where $X^{*}$ is the dual of $X$.

To prove group theoreticity, we use the following key characterization:
\begin{proposition}\cite[Corollary 4.14]{DGNO}.
A modular category $\CC$ is group theoretical if and only if it is integral and there is a symmetric subcategory $\LL$ such that $\mZ_\CC(\LL)_{ad} \subset \LL$ 
\end{proposition}

Now consider an integral $N$-metaplectic modular category. There are three cases to consider depending on the value of $N$. The relevant fusion rules for each of these cases are listed below (note: fusion rules for $\sqrt{N}$ and $\sqrt{\frac{N}{2}}$ dimensional objects are not included, but can be found here \cite{ACRW,BGPR}). 

As stated in \cite{GRR}, the unitary modular category $SO(N)_2$ for odd $N >$ 1 has two simple objects, $X_1 , X_2$ of dimension $\sqrt{N}$, two simple objects \1, $Z$ of dimension 1, and $\frac{N-1}{2}$ objects $Y_i , i = 1, ... , \frac{N-1}{2}$ of dimension 2 \cite{ACRW}. Observe that, in order for a metaplectic modular category to be integral, N must be a perfect square. Observe that, in order for a metaplectic modular category to be integral, $N$ must be a perfect square. 
The fusion rules are \cite{ACRW}:

\begin{enumerate}
\item $Z \ot Y_{i} \cong Y_{i}, Z^{\ot2} \cong \1$

\item $Y_{i} \ot Y_{j} \cong Y_{min\{i+j, N-i-j\}} \oplus Y_{|i-j|}, $ for $i \neq j$ 

\item $Y_{i}^{\ot2} = \1 \oplus Z \oplus Y_{min\{2i,N-2i\}}$.
\end{enumerate}

 The unitary modular category $SO(N)_2$ where $N \equiv 2$ (mod 4) has rank $k+7$, where $2k = N$. The group of isomorphism classes of invertible objects is isomorphic to $\Z_4$. We will denote by $g$ a generator of this group, and will abuse notation referring to the invertible objects as $g^j$. There are $k-1$ self-dual simple objects, $Y_i$, of dimension 2. The remaining four simples, $V_i$, have dimension $\sqrt{k}$. The fusion rules re-written from \cite{GRR} are:

\begin{enumerate}
\item $g \ot Y_{i} \cong Y_{k-i}, g^2 \ot Y_{i} \cong Y_{i}$

\item $Y_{i}^{\ot2} \cong \1 \oplus g^2 \oplus Y_{min \{2i, 2k-2i\}}$

\item $Y_{i} \ot Y_{j} \cong Y_{min\{i+j, 2k-i-j\}} \oplus Y_{|i-j|},$ when $i + j \neq k$

\item $Y_{i} \ot Y_{j} \cong g \oplus g^3 \oplus Y_{|i-j|}$ when $i + j = k$.

\end{enumerate}

The modular category $SO(N)_2$ with $N \equiv 0$ (mod 4) where $2k = N$, has rank $k+7$ and dimension $4N$ \cite{GRR}. The simple objects have dimension 1, 2, and $\sqrt{k}$. All simple objects are self-dual, and the $Y_i$ have dimension 2, while $V_i$, $W_i$ have dimension $\sqrt{k}$. For $k > 2$ the key fusion rules re-written from \cite{GRR} and are as follows: 

\begin{enumerate}

\item $g \ot Y_{i} \cong f \ot Y_{i} \cong Y_{k-i}$

\item $Y_{i}^{\ot2} \cong \1 \oplus f \oplus g \oplus fg $, when $i = \frac{k}{2}$

\item $Y_{i}^{\ot2} \cong \1 \oplus fg \oplus Y_{min \{2i, 2k-2i\}}$, when $i \neq \frac{k}{2}$

\item $Y_{i} \ot Y_{j} \cong Y_{min\{i+j, 2k-i-j\}} \oplus Y_{|i-j|}$, when $i + j \neq k$

\item $Y_{i} \ot Y_{j} \cong g \oplus f \oplus Y_{|i-j|}$, when $i + j = k$.
\end{enumerate}

\begin{proposition}
Let $\CC$ be a braided tensor category. $\CC$ is symmetric if and only it it coincides with its symmetric center $\mZ_{\CC} (\CC)$ \cite{M}.
\end{proposition}

\begin{proposition}
Let $\CC$ be a modular category and $\LL \subset \CC$ be a semisimple tensor subcategory.  Then $dim(\LL) dim(\mZ_{\CC} (\LL)) = dim(\CC)$ \cite{M}.
\end{proposition}

\begin{definition}
The balancing equation $\theta_{i}\theta_{j}S_{ij} = \sum_{k \in Irr(\CC)}N^k_{i^{*}j}d_k\theta_k$
\end{definition}

\begin{lemma}
Every integral metapletic modular category has a symmetric subcategory $\LL$.  
\begin{enumerate}[(i)]
\item For $N$ odd, $\LL$ is  generated by $\1, Z$ and $Y_{it}$ where $t = \sqrt{N}$ and $1 \leq i \leq \frac{t-1}{2}$.  
\item For $N \equiv 2$ $(mod \text{ } 4)$, $\LL$ is generated by $\1, g^2$ and $Y_{2ln}$ where $l = \sqrt{k}$ and $1 \leq n \leq \frac{l-1}{2}$.  
\item For $N \equiv 0$ $(mod \text{ } 4)$, $\LL$ is generated by $\1, f, g, fg$ and $Y_{2lm}$ where $1 \leq m \leq \frac{l - 2}{2}$.
\end{enumerate}
\end{lemma}
\begin{proof}
First, we examine the case of $N$ odd in detail. We need to show that $\LL$ is symmetric, that is, by Proposition 2.5, $\LL = \mZ_{\LL} (\LL)$.  By Proposition 2.6, $dim(\mZ_{\CC} (\LL)) = 2t$.  

There is a faithful $\Z_2$ grading on $\CC$:  $$\CC_{\1} = \{\1, Z, Y_i\} \\ \CC_{Z} = \{X_i\}$$ We know that $\CC_{pt} \subset \LL$.  Applying the fact that $\mZ_{\CC} (C_{pt}) = \CC_{\1}$ \cite{GN}, we know that $\mZ_{\CC} (\LL) \subset \CC_{\1}$.  

The object Z is a boson \cite[Lemma 3.3]{ACRW}.  We perform the de-equivariantization of $\CC$ by $\langle Z \rangle \cong \Rep(\Z_2)$ as the trivial component of the de-equivariantization $D_0$, the image of $\CC_{\1}$, is isomorphic to $\Z_{N}$ \cite[Lemma 3.4]{ACRW}.  The de-equivariantization functor maps $ \1$ and $Z$ to $0$, and each $Y_i$ to the simple objects $i \oplus -i \in \Z_{N}$.  The subcategory tensor generated by $Y_i$ maps to $\langle i \rangle \in \Z_{N}$.  

As $\langle t \rangle$ is the only subgroup of order t, it is the image of the only subcategory of dimension $2t$, which must be $\mZ_{\CC} (\LL)$.  Now it is clear $\mZ_{\CC} (\LL)$ is the subcategory tensor generated by $Y_t$. Because $\LL$ is equal to its centralizer in $\CC$, it is symmetric.

The proof for the case of $N \equiv 2$ (mod 4) proceeds in the same fashion. By Proposition 2.6, $dim(\mZ_{\CC} (\LL)) = 4\ell$.  The faithful grading is by $\Z_4$:
\begin{align*} 
\CC_{\1} &= \{\1, g^2, Y_{i}\} \text{ where $i$ is even} \\
\CC_{g} &= \{ V_1, V_4 \} \\
\CC_{g^2} &= \{g, g^3, Y_i\} \text{ where $i$ is odd} \\
\CC_{g^3} &= \{V_2, V_3\}. 
\end{align*}
We know that $g^2$ is a boson \cite[Lemma 3.3]{BPR}.  As $g^2 \in \LL$, $\mZ_{\CC} (\LL) \subset \mZ_{\CC} (\langle g^2 \rangle)$. A dimension counting argument tells us that $\mZ_{\CC}(\langle g^2 \rangle)$ must be $\CC_{\1} \cup \CC_{g^2}$. Forming the de-equivariantization by $g^2$, the image of $\mZ_{\CC} (\LL)$ is contained within $\D_0$. Following the proof from \cite[Lemma 3.3]{ACRW} $\D_0 \cong \Z_N$, so it only contains the image of one subcategory of the correct size and $\mZ_{\CC} (\LL)$ is the subcategory tensor generated by $Y_l$.  Therefore, $\mZ_{\LL} (\LL) = \LL$ and $\LL$ is symmetric.

The proof for the case of $N \equiv 0$ (mod 4) proceeds in exact fashion as the proof for $N \equiv 2$ (mod 4), except in this case the faithful grading is by $\Z_2 \times \Z_2$,
\begin{align*} 
\CC_{\1} &= \{\1, f, g, fg, Y_{i}\} \text{ where $i$ is even} \\
\CC_{g} &= \{ V_1, V_2 \} \\
\CC_{f} &= \{W_1, W_2\}\\
\CC_{fg} &= \{Y_i\} \text{ where $i$ is odd}. 
\end{align*}
we de-equivariantize by the boson $fg$ \cite[Lemma 4.3]{BGPR}, and apply the balancing equation and dimension counting to show that $Y_i \in \mZ_{\CC} (fg)$,  the preimage of $\D_0$, for all $i$.

\end{proof}

\begin{theorem}
Integral metaplectic modular categories are group-theoretical.
\end{theorem}\label{thm:GT}

\begin{proof}

We examine the symmetric subcategory $\LL$ we defined for each case. Using the fusion rules for each case and $\mZ_{\CC}(\LL)$, which we found in the proof of our previous result, it is clear that $(\mZ_{\CC}(\LL))_{ad} \subset \LL$. Therefore, in each case $\CC$ is group theoretical.

\end{proof}

\section{Categories With Fusion Rules of $SO(8)_2$}

Categories with the fusion rules of $SO(8)_2$ present an interesting special case.  In these categories, objects $V_1, V_2, W_1, W_2$ of dimension $\sqrt{\frac{N}{2}}$ have the same dimension as objects $Y_i$, creating extra symmetry.  

\begin{definition}
A category $\CC$ is \textbf{group theoretical} if and only if its Drinfeld center $\mZ (\CC)$ is equivalent to the representation category of the twisted double of some finite group $G$:
    $$\mZ (\CC) \cong \text{Rep }D^{\omega}(G)$$
\end{definition}

Having determined that categories with the fusion rules of $SO(8)_2$ are group theoretical, it is natural to ask which finite group they \textquotedblleft come from."  Applying the facts that $|G|^2 = \text{dim(Rep}D^\omega (G))= \text{dim} \mZ (\CC) = (\text{dim} \CC)^2 = 1024$ and $\text{rank} (\text{Rep}(D^\omega {G})) = \text{rank} (\mZ (\CC)) = (\text{rank} \CC )^2 = 256$,  we search the database of modular data for twisted Drinfeld double of finite groups of order less than 47 created by Angus Gruen \cite{AG} for pairs $(\omega, G)$ that fit these restrictions.  We determined that categories with the fusion rules of $SO(8)_2$ come from $G =$[32, 49], the inner holomorph of the dihedral group of order 8.  Notably, there are 72 Morita-equivalence classes of categories for which $\text{Rep}(D^{\omega}G)$ has the correct dimension and rank, and, for each of these $\omega$ is a nontrivial cocyle.  

\section{Link Invariants}
Since every topological quantum computation is equivalent to evaluating the link invariant associated with the anyonic system on some braid, another natural inquiry is to determine the link invariant associated with categories with fusion rules of $SO(8)_2$. The 2-variable Kauffman polynomial is known to be associated with  the quantum group $U_{q,so(n)}$.  Examining the work of Tuba and Wenzl, we know the invariant associated with these categories is Wenzl's parameterization of the 2-variable Kauffman polynomial, evaluated at $q = e^{\frac{i \pi}{8}}, r = -q^{-1}$ \cite{TW}.  The extended Property \textbf{F} conjecture predicts that the link invariant associated with a braided fusion category $\CC$ with Property \textbf{F} is classical; therefore, we expect this evaluation to produce a classical invariant.  

In the following definitions, $D$ is the diagram of a link. $D_+$, $D_-$, $D_0$, and $D_\infty$ are identical diagrams except near a point where they are \\

\begin{figure}[H]
\centering
    \begin{subfigure}{10 mm}
        \includegraphics[width=\textwidth]{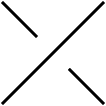}
        \caption*{$D_+$}
    \end{subfigure}
    \hspace{10mm}
    \begin{subfigure}{10 mm}
        \includegraphics[width=\textwidth]{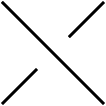}
        \caption*{$D_-$}
    \end{subfigure}
    \hspace{10mm}
    \begin{subfigure}{10 mm}
        \includegraphics[width=\textwidth]{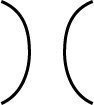}
        \caption*{$D_0$}
    \end{subfigure}
    \hspace{10mm}
     \begin{subfigure}{11 mm}
        \includegraphics[width=\textwidth]{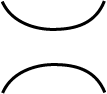}
        \caption*{$D_{\infty}$}
    \end{subfigure}
\end{figure}

\begin{definition}
The semi-oriented Kauffman polynomial is a function $$K: \{\text{Oriented links in } S^3\} \rightarrow \Z [a^{\pm 1}, z^{\pm 1}]$$
where $\Z [a^{\pm 1}, z^{\pm 1}]$ is the Laurent polynomial over $a$ and $z$.
It is uniquely defined by $K(L) = a^{-w(D)} \tilde{K} (D)$ where $$\tilde{K} :  \{\text{Oriented diagrams of links}\} \rightarrow \Z [a^{\pm 1}, z^{\pm 1}]$$ is uniquely determined by:
\begin{enumerate}[(I)]
    \item $\tilde{K} (\bigcirc) = 1$,
    \item $\tilde{K}$ is invariant under regular isotopy,
    \item $a^{-1}\tilde{K}\bigl(\begin{smallmatrix}
\includegraphics[width = 2.5 mm]{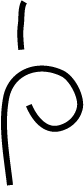}
\end{smallmatrix} \bigr) = \tilde{K} (D) = a\tilde{K} \bigl(\begin{smallmatrix}
\includegraphics[width = 2.5 mm]{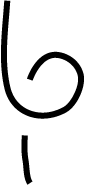}
\end{smallmatrix} \bigr)$
    \item $\tilde{K} (D_+) + \tilde{K} (D_-) = z(\tilde{K} (D_0) + \tilde{K} (D_{\infty})$
\end{enumerate}
\end{definition}

\begin{definition}
Wenzl's parameterization of the Kauffman polynomial is $$W : \{\text{Oriented links in } S^3\} \rightarrow \Z [r^{\pm, 1}, q^{\pm 1}]$$ where $\Z [q^{\pm 1}, r^{\pm 1}]$ is the Laurent polynomial over $r$ and $q$. It is uniquely defined by $W(L) = r^{-w(D)} \tilde{W}(D)$ where 
$$\tilde{W} :  \{\text{Oriented diagrams of links}\} \rightarrow \Z [r^{\pm 1}, q^{\pm 1}]$$ is uniquely determined by: 
\begin{enumerate}[(i)]
  \item $\tilde{W}(\bigcirc) = 1$
    \item $\tilde{W}$ is invariant under regular isotopy. 
    \item $r\tilde{W}\bigl(\begin{smallmatrix}
\includegraphics[width = 2.5 mm]{twist1.png}
\end{smallmatrix} \bigr) = \tilde{W}(|) = r^{-1}\tilde{W} \bigl(\begin{smallmatrix}
\includegraphics[width = 2.5 mm]{twist2.png}
\end{smallmatrix} \bigr)$
  \item $\tilde{W}(D_{-}) - \tilde{W}(D_{+}) = (q-q^{-1})(\tilde{W}(D_0) - \tilde{W}(D_{\infty})$
\end{enumerate}
\end{definition}

\begin{definition}
The Dubrovnik parameterization of the Kauffman polynomial is $$F : \{\text{Oriented links in } S^3\} \rightarrow \Z [r^{\pm, 1}, q^{\pm 1}]$$ where $\Z [q^{\pm 1}, r^{\pm 1}]$ is the Laurent polynomial over $r$ and $q$. It is uniquely defined by $F(L) = r^{-w(D)} \tilde{F}(D)$ where 
$$\tilde{F} :  \{\text{Oriented diagrams of links}\} \rightarrow \Z [\alpha^{\pm 1}, \omega^{\pm 1}]$$ is uniquely determined by: 
\begin{enumerate}[(i)]
  \item $\tilde{F}(\bigcirc) = 1$
    \item $\tilde{F}$ is invariant under regular isotopy. 
    \item $\alpha^{-1}\tilde{F}\bigl(\begin{smallmatrix}
\includegraphics[width = 2.5 mm]{twist1.png}
\end{smallmatrix} \bigr) = \tilde{F}(|) = \alpha\tilde{F} \bigl(\begin{smallmatrix}
\includegraphics[width = 2.5 mm]{twist2.png}
\end{smallmatrix} \bigr)$
  \item $\tilde{F}(D_{+}) - \tilde{F}(D_{-}) = \omega(\tilde{F}(D_0) - \tilde{F}(D_{\infty})$
\end{enumerate}
\end{definition}

\begin{theorem}
The link invariant associated with categories with the same fusion rules as $SO(8)_2$ is classical and given by $$[W(L)]_{(r, q) = (-\frac{1}{q}, e^{\frac{i \pi}{8}})} = \frac{1}{2} \sum_{X \subset L} i^{\text{linking number} (X, L-X)}$$
\end{theorem}

\begin{proof}
As $N=8$, we know that the link invariant associated with these categories is $W(L)$ with $r=q^7$ \cite{TW}.  Lemma 7.5 of \cite{TW} shows that Burman-Murikami-Wenzl algebras $BMW(q,q)$, the algebras associated with the 2-variable Kauffman polynomial, do not depend on the particular value of q.  Choosing $q=e^{\frac{i \pi}{8}}$, we have $r= -\frac{1}{q}$.

The only differences between the various parameterizations of the two-variable Kauffman polynomial lie in their twist and skein relations.  It is easy to see $$W(L) = [F(L)]_{(\alpha, \omega) = (r^{-1}, -(q-q^{-1}))}$$ Applying \cite[Theorem 1]{LICK}, then 
\begin{align*}
[W(L)]_{(r, q) = (-\frac{1}{q}, e^{\frac{i \pi}{8}})} &= (-1)^{c(L) - 1}[K(L)]_{(a,z) = (i r^{-1}, i(q-q^{-1}))} \\&= (-1)^{c(L) - 1}[K(L)]_{(a,z) = (-q^5, q^5 + (q^5)^{-1})}
\end{align*}

Then, taking an expression for this evaluation of $K(L)$ of \cite[Table 1]{LICK} we have,
$$[W(L)]_{(r, q) = (-\frac{1}{q}, e^{\frac{i \pi}{8}})} = \frac{1}{2} \sum_{X \subset L} i^{\text{linking number} (X, L-X)}$$
which is quite classical indeed.
\end{proof}

\end{document}